\documentclass[article,12pt]{elsarticle}

\usepackage{amssymb}
\usepackage{graphicx}
\usepackage{amsmath}
\usepackage{booktabs}
\usepackage{bbm}

\newtheorem{priteo}{Theorem}
\newtheorem{lema}{Lemma}
\newtheorem{coro}{Corollary}

\newtheorem{defi}{Definition}
\newtheorem{prop}{Proposition}
\newenvironment{proof}[1][Proof]{\textbf{#1.} }{\ \rule{0.5em}{0.5em}}

\newcommand{\RR}{{\rm\bf R}}

\newcommand{\seta}{\longrightarrow}
\newcommand{\dpt}{\displaystyle}

\newcommand{\fix}{{\rm Fix}}
\newcommand{\one}{\mathbbm{1}}

\journal{Journal of Mathematical Analysis and Applications}
\begin{document}

\begin{frontmatter}
\title{Periodic solutions in an array of coupled FitzHugh-Nagumo cells}
\markboth{ Torus \today}{ Torus \today}

\author{Isabel Salgado Labouriau}
\address{Centro de Matem\'atica da Universidade do Porto.\\ Rua do Campo Alegre 687, 4169-007 Porto, Portugal}
\author{Adrian C. Murza}

\address{Departamento de Matem\'atica, Faculdade de Ci\^encias\\
Universidade do Porto.\\ Rua do Campo Alegre 687, 4169-007 Porto, Portugal}

\date{\today}
\begin{abstract}
We analyse the dynamics of an array of $N^2$ identical cells coupled in the shape of a torus.
Each cell is a 2-dimensional ordinary differential equation of FitzHugh-Nagumo type
and the total system is $\mathbb{Z}_N\times\mathbb{Z}_N-$symmetric.
The possible  patterns  of oscillation, compatible with the symmetry, are described.
The types of patterns that effectively arise through Hopf bifurcation are shown to depend on the signs of the coupling constants,  under conditions ensuring that the equations have only one equilibrium state.
\end{abstract}

\end{frontmatter}

\maketitle

\section{Introduction}
Hopf bifurcation has been intensively studied in equivariant dynamical systems in the recent years from both theoretical and applied points of view. Stability of equilibria, synchronization of periodic solution and in general oscillation patterns, stability of the limit cycles that arise at the bifurcation point are among the phenomena whose analysis is related to the Hopf bifurcation in these systems. Periodic solutions arising in systems with dihedral group symmetry were studied by Golubitsky et al. \cite{GS88} and Swift \cite{GSdihedral},
 Dias and Rodrigues  \cite{Dias_symm} dealt with the symmetric group, 
 Sigrist  \cite{Sigrist} with the orthogonal group, to cite just a few of them. 
 Dias {\em et al.}  \cite{Dias_int} studied  periodic solutions in coupled cell systems with internal symmetries, while Dionne extended the analysis to Hopf bifurcation in equivariant dynamical systems with wreath product \cite{Dionne_wreath} and direct product groups \cite{Dionne_direct}.
The general theory of patterns of oscillation arising in systems with abelian symmetry was developed by Filipsky and Golubitsky \cite{AbelianHopf}.
The dynamical behavior of 1-dimensional ordinary differential equations coupled in a square array,  of arbitrary size $(2N)^2$, with the symmetry
$\mathbf{D}_4\dotplus\left(\mathbb{Z}_N\times\mathbb{Z}_N\right)$, was studied by Gillis and  Golubitsky \cite{Gillis}.

In this paper we use a similar idea to that of \cite{Gillis} to describe arrays of  $N^2$ cells where each cell is represented by a subsystem that is a  2-dimensional differential equation of FitzHugh-Nagumo  type.
We are interested in the periodic solutions arising at  a first Hopf bifurcation from the fully synchronised equilibrium.
To each equation in the array we add a coupling term that describes how each cell is affected by its neighbours.
The coupling may be associative, when it tends to reduce the difference between consecutive cells, 
or dissociative, when differences are increased.
For associative coupling we find, not surprisingly, bifurcation into a stable periodic solution where all the cells are synchronised with  identical behaviour. 

When the coupling is dissociative in either one or both directions, the first Hopf bifurcation gives rise to rings of $N$ fully synchronised cells.
All the rings oscillate with the same period, with a $\frac{1}{N}$-period phase shift between rings.
When there is one direction of associative coupling, the synchrony rings are organised along it.
Dissociative coupling in both directions yields rings organised along the diagonal.
The stability of these periodic solutions was studied numerically and were found to be unstable for small numbers  of cells, stability starts to appear at $N\geqslant 11$.

For all types of coupling, there are further Hopf bifurcations, but these necessarily yield unstable solutions.

\medbreak

This paper is organised as follows.
The equations are presented in section~\ref{onetorus_theor} together with their symmetries.
Details about the action of the symmetry group $\mathbb{Z}_N\times\mathbb{Z}_N$ are summarised in
section~\ref{sectionSymmetry}:
we identify the $\mathbb{Z}_N\times\mathbb{Z}_N-$irreducible subspaces of
$\mathbb{R}^{2N^2}$; the isotypic components; isotropy subgroups and their fixed point subspaces  for this action. This allows, in section~\ref{SecSubgroups}, the study of the Hopf bifurcation with symmetry
$\mathbb{Z}_N\times\mathbb{Z}_N$, applying the abelian Hopf bifurcation theorem  \cite{AbelianHopf} to identify the symmetries of the branch of small-amplitude peridic solutions that may bifurcate from equilibria.
In section~\ref{Stability} we derive the explicit expression of the $2N^2$ eigenvectors and eigenvalues of the system linearised about the origin.
Next, in section~\ref{secBifciszero} we perform a detailed analysis on the Hopf bifurcation
by setting a parameter $c$ to zero.
In this case the FitzHugh-Nagumo equation reduce to  a Van der Pol-like equation. 
Finally, in section~\ref{seccsmall}, we characterise the bifurcation conditions for $c>0$ small.

\section{Dynamics  of FitzHugh-Nagumo coupled in a torus
and its symmetries}\label{onetorus_theor}

The building-blocks of our square array are the following $2-$dimensional ordinary differential equations of
FitzHugh-Nagumo (FHN) type
\begin{equation}\label{FHN}
\begin{array}{ll}
        \dot{x}=x\left(a-x\right)\left(x-1\right)-y&=f_1(x,y)\\
        \dot{y}=bx-cy&=f_2(x,y)
    \end{array}
\end{equation}
where $a,b,c\geqslant 0$.
Consider a system of $N^2$ such equations, coupled as a discrete torus:
\begin{equation}\label{nFHN_2d}
    \begin{array}{l}
        \dot{x}_{\alpha,\beta}=
        x_{\alpha,\beta}\left(a-x_{\alpha,\beta}\right)\left(x_{\alpha,\beta}-1\right)
        -y_{\alpha,\beta}+\gamma (x_{\alpha,\beta}-x_{\alpha+1,\beta})+\delta (x_{\alpha,\beta}-x_{\alpha,\beta+1})\\
        \dot{y}_{\alpha,\beta}=bx_{\alpha,\beta}-cy_{\alpha,\beta}
    \end{array}
\end{equation}
where $\gamma\ne\delta$ and $1\leqslant\alpha\leqslant N$,  $1\leqslant\beta\leqslant N$, with both
$\alpha$ and $\beta$ computed $\pmod{N}$. 
When either $\gamma$ or $\delta$ is negative, we say that the coupling is {\em associative}: the coupling term tends to reduce the difference to the neighbouring cel, 
otherwise we say the coupling is {\em dissociative}.
We restrict ourselves to the case where $N\geqslant3$ is prime. 

The coupling structure in \eqref{nFHN_2d} is similar, but not identical, to the general case studied by
Gillis and  Golubitsky \cite{Gillis}.
There are two main differences: first, they consider an arbitrary {\em even} number of cells, whereas we study a prime number of cells. 
The second difference is that cells in \cite{Gillis} are bidirectionally coupled, and the coupling in \eqref{nFHN_2d} is unidirectional. 
These differences will be reflected in the symmetries of \eqref{nFHN_2d}.

The first step in our analysis consists in describing the symmetries of (\ref{nFHN_2d}).
Our phase space is
$$
\mathbb{R}^{2N^2}=\left\{\left(x_{\alpha,\beta},y_{\alpha,\beta}\right)|~~1\leqslant \alpha,\beta\leqslant N,~x_{\alpha,\beta},y_{\alpha,\beta}\in\mathbb{R}\right\}
$$
and (\ref{nFHN_2d}) is equivariant under the cyclic permutation of the collumns in the squared array:
\begin{equation}\label{ActionGamma1}
        \gamma_1 (x_{1,\beta},\ldots,x_{N,\beta};y_{1,\beta},\ldots,y_{N,\beta})=
        (x_{2,\beta},\ldots,x_{N,\beta},x_{1,\beta};y_{2,\beta},\ldots,y_{N,\beta},y_{1,\beta})
\end{equation}
as well as under the cyclic permutation of the rows in the squared array:
\begin{equation}\label{ActionGamma2}
        \gamma_2 (x_{\alpha,1},\ldots,x_{\alpha,N};y_{\alpha,1},\ldots,y_{\alpha,N})=
        (x_{\alpha,2},\ldots,x_{\alpha,N},x_{\alpha,1};y_{\alpha,2},\ldots,y_{\alpha,N},y_{\alpha,1}).
\end{equation}
Thus, the symmetry group of (\ref{nFHN_2d}) is the group generated by $\gamma_1$ and $\gamma_2$ denoted
$\mathbb{Z}_N\times\mathbb{Z}_N=\langle\gamma_1,\gamma_2\rangle$.
Note that, since the coupling in \eqref{nFHN_2d} is unidirectional and since the coupling constants 
$\gamma$ and $\delta$ are not necessarily equal, there is no addtional symmetry, like the $\mathbf{D}_4$ in \cite{Gillis}.
Indeed, we will show that the case $\gamma=\delta$ is degenerate.
We will use the notation
$\gamma_1^r\cdot\gamma_2^s\in\mathbb{Z}_N\times\mathbb{Z}_N$ as $(r,s)=\gamma_1^r\cdot\gamma_2^s.$
We refer to the system \eqref{nFHN_2d} in an abbreviated form as either $\dot{z}=f(z)$,
$z=\left(x_{\alpha,\beta},y_{\alpha,\beta}\right)$ or $\dot{z}=f(z,\lambda)$ where $\lambda\in\mathbb{R}$
is a bifurcation parameter to be specified later.
The compact Lie group $\mathbb{Z}_N\times\mathbb{Z}_N$ acts linearly on $\mathbb{R}^{2N^2}$
and $f$  commutes with it (or is $\mathbb{Z}_N\times\mathbb{Z}_N-$equivariant).

We start  by recalling some definitions from \cite{GS88} adapted to our case.

The isotropy subgroup $\Sigma_z$ of $\mathbb{Z}_N\times\mathbb{Z}_N$ at a point $z\in\mathbb{R}^{N^2}$ is defined to be
$$
\Sigma_z=\left\{(r,s)\in\mathbb{Z}_N\times\mathbb{Z}_N:(r,s)\cdot z=z\right\}.
$$

Moreover, the fixed point subspace of a subgroup $\Sigma\in\mathbb{Z}_N\times\mathbb{Z}_N$ is
$$
\mathrm{Fix}\left(\Sigma\right)=\left\{z\in\mathbb{R}^{2N^2}:(r,s)\cdot z=z,\ \forall (r,s)\in\Sigma\right\}
$$
and $f\left(\mathrm{Fix}\left(\Sigma\right)\right)\subseteq\mathrm{Fix}\left(\Sigma\right)$.

\begin{defi}\label{defi irred}
Consider a group $\mathbf{\Gamma}$
acting linearly on
$\mathbb{R}^{n}.$ Then
\begin{enumerate}
\item A subspace $\mathbf{V}\subseteq\mathbb{R}^{n}$ is said {\em $\mathbf{\Gamma}-$invariant}, if $\sigma\cdot v\in \mathbf{V},~\forall\sigma\in\mathbf{\Gamma},~\forall v\in \mathbf{V};$
\item  A subspace $\mathbf{V}\subseteq\mathbb{R}^{n}$ is said {\em $\mathbf{\Gamma}-$irreducible} if it is $\mathbf{\Gamma}-$invariant and if the only $\mathbf{\Gamma}-$invariant subspaces of $\mathbf{V}$ are $\left\{\mathbf{0}\right\}$ and $\mathbf{V}$.
\end{enumerate}
\end{defi}

\begin{defi}
Suppose a group $\Gamma$ acts on two vector spaces $\mathbf{V}$ and $\mathbf{W}$.
We say that $\mathbf{V}$ is $\Gamma$-isomorphic to $\mathbf{W}$ if there exists a linear isomorphism
$A:\mathbf{V}\seta \mathbf{W}$ such that $A(\sigma x)=\sigma A(x)$ for all $x\in \mathbf{V}$.
If  $\mathbf{V}$ is {\em not} $\Gamma$-isomorphic to $\mathbf{W}$ we say that they are distinct representations of
$\Gamma$.
\end{defi}

\section{The $\mathbb{Z}_N\times\mathbb{Z}_N$ action}\label{sectionSymmetry}

The action of $\Gamma=\mathbb{Z}_N\times\mathbb{Z}_N$ is identical in the $x_{\alpha,\beta}$ and the $y_{\alpha,\beta}$ coordinates, i.e.
 $\Gamma$ acts diagonally,
$\gamma(x,y)=(\gamma x,\gamma y)$ in $\mathbb{R}^{2N^2}$ for $x,y\in \mathbb{R}^{N^2}$.
Hence,
instead of taking into account the whole set of $\left(x_{\alpha,\beta},y_{\alpha,\beta}\right)\in\mathbb{R}^{2N^2},$ we will partition it into two subspaces,
$\mathbb{R}^{N^2}\times\{\mathbf{0}\}$ and $\{\mathbf{0}\}\times\mathbb{R}^{N^2}$,
namely $x_{\alpha,\beta}\in\mathbb{R}^{N^2}$ and $y_{\alpha,\beta}\in\mathbb{R}^{N^2}$.
The action of $\mathbb{Z}_N\times\mathbb{Z}_N$ on $\mathbb{R}^{N^2}$ has been studied by Gillis and Golubitsky in~\cite{Gillis} for $N=2n$, we adapt their results to our case.

Let $\mathbf{k}=\left(k_1,k_2\right)\in\mathbb{Z}^2$ and consider the subspace
${\mathbf{V_k}}\subset\mathbb{R}^{N^2}$, where $ \left( x_{\alpha,\beta} \right)\in \mathbf{V_k}$ if and only if
\begin{equation}\label{irreducible_space}
x_{\alpha,\beta} =
{\mathrm{Re}\left(z\exp\left[\frac{2\pi i}{N}\left(\alpha,\beta\right)\cdot\mathbf{k}\right]\right)\in\mathbb{R}^{N^2}}
\ {z\in\mathbb{C},\ 1\leqslant \alpha,\beta\leqslant N} \ .
\end{equation}

\begin{prop}\label{PropSubspaces}
Consider the action of $\mathbb{Z}_N\times\mathbb{Z}_N$  on $\mathbb{R}^{N^2}$ given
in~\eqref{ActionGamma1} and \eqref{ActionGamma2} with $N$ prime and let $I$ be the set of indices $\mathbf{k}=\left(k_1,k_2\right)$ listed in Table~\ref{table 1}.
Then for $\mathbf{k}\in I$ we have
\begin{enumerate}
\item\label{prop irred}
$\mathrm{dim}\mathbf{V_k}=2$ except for  $\mathbf{k}=\mathbf{0},$ where $\mathrm{dim}\mathbf{V_{0}}=1$.
\item\label{lemma_invariance}
 Each $\mathbf{V_k}$ defined in \eqref{irreducible_space} is
 $\mathbb{Z}_N\times\mathbb{Z}_N-$invariant and  $\mathbb{Z}_N\times\mathbb{Z}_N-$irreducible.
 \item\label{lemma_nonisomorphic}
The subspaces $\mathbf{V_k}$ are all distinct $\mathbb{Z}_N\times\mathbb{Z}_N$ representations.
\item\label{eqproofLema1b}
The group element $\left(r,s\right)$ acts on $\mathbf{V_k}$ as a rotation:
$$
\left(r,s\right)\cdot{z}=\mathrm{exp}\left[\frac{2\pi i}{N}\left(r,s\right)\cdot \mathbf{k}\right]z .
$$
\item\label{eqLema1}
The subspaces $\mathbf{V_k}$ verify
$\bigoplus_{\mathbf{k}\in I}\mathbf{V_k}=\mathbb{R}^{N^2}$.
\item\label{IsotropySubgroups}
The non-trivial isotropy subgroups for  $\mathbb{Z}_N\times\mathbb{Z}_N$ on $\mathbb{R}^{N^2}$ are
 $\widetilde{\mathbb{Z}}_N\left(r,s\right)$, the subgroups generated by one element
$\left(r,s\right)\neq \left(0,0\right)$.
\item\label{FixedPointSubspaces}
If $\left(r,s\right)\neq \left(0,0\right)$ then
$$
\fix\left(\widetilde{\mathbb{Z}}_N\left(r,s\right)\right)=\sum_{\stackrel{\mathbf{k}\cdot(r,s)=0}{\pmod{N}}}\mathbf{V_k}
\qquad\mbox{and}\qquad
\dim\fix\left(\widetilde{\mathbb{Z}}_N\left(r,s\right)\right)=N.
$$
\end{enumerate}
\end{prop}

\begin{table}[h]
\centering
\begin{center}
\caption{Types of $\mathbb{Z}_N\times\mathbb{Z}_{N}-$irreducible representations in
$\mathbf{V}_{\mathbf{k}}\in\mathbb{R}^{N^2},$ where $\mathbf{V}_{\mathbf{k}}$ is the subspace of
$\mathbb{R}^{N^2}$ corresponding to
$\mathbf{k}=(k_1,k_2)\in I\subset\mathbb{Z}^2$.}\label{table 1}
\end{center}
\begin{tabular}{ccccc}
\toprule
Type & $\mathrm{dim}\left(\mathbf{V_k}\right)$ & $\mathbf{k}$ & Restrictions\\
\midrule
\centering
$\left(1\right)$ & $1$ & $\left(0,0\right)$ &\\
$\left(2\right)$ & $2$ & $\left(0,k_2\right)$ &  $1\leqslant k_2 \leqslant \left(N-1\right)/2$\\
$\left(3\right)$ & $2$ & $\left(k_1,0\right)$ &  $1\leqslant k_1 \leqslant \left(N-1\right)/2$\\
$\left(4\right)$ & $2$ & $\left(k_1,k_1\right)$ &  $1\leqslant k_1 \leqslant \left(N-1\right)/2$\\
$\left(5\right)$ & $2$ & $\left(k_1,k_2\right)$ &  $1\leqslant k_2<k_1 \leqslant N-1$\\
\bottomrule
\end{tabular}
\end{table}

\begin{proof}
The arguments  given in~\cite[Lemma 3.1]{Gillis} with suitable adaptations show that statements \ref{prop irred}.--\ref{eqproofLema1b}. hold and also that  the subspaces $\mathbf{V_k}$ with $\mathbf{k}\not\in I$ are redundant.
Since the $\mathbf{V_k}$ are all distinct irreducible representations, a calculation
using~\ref{prop irred}. shows that
$
\sum_{\mathbf{k}\in I}\mathrm{dim}\mathbf{V_k}=N^2
$, establishing~\ref{eqLema1}.

For
 \ref{IsotropySubgroups}.  note that
since $N$ is prime, the only non trivial subgroups of $\mathbb{Z}_N\times\mathbb{Z}_N$ are
the cyclic subgroups $\widetilde{\mathbb{Z}}_N\left(r,s\right)$ generated by
$\left(r,s\right)\neq \left(0,0\right)$.
Each $\left(r,s\right)$ fixes the elements of
$\mathbf{V}_{\left(k_1,k_2\right)}$ when $\left(k_1,k_2\right)=\left(N-s,r\right)$.
Using~\ref{eqproofLema1b}. it follows that $\left(r,s\right)\neq \left(0,0\right)$ fixes
 $x=\left(x_{\alpha,\beta}\right)\in\mathbf{V}_{\left(k_1,k_2\right)}$ with $x\ne 0$
if and only if
$
\mathrm{exp}\left[\frac{2\pi i}{N}\left(r,s\right)\cdot\left(k_1,k_2\right)\right]=1
$
i.e. if and only if
$ \left(r,s\right)\cdot\left(k_1,k_2\right)=0\pmod{N}$.
Thus $\mathrm{Fix\left(r,s\right)}$ is the sum of all the subspaces
$\mathbf{V}_{\left(k_1,k_2\right)}$ such that $ \left(r,s\right)\cdot\left(k_1,k_2\right)=0\pmod{N}$, it remains to compute its dimension.

Let $\left(k_1^\prime,k_2^\prime\right)=\left(N-s,r\right)$, so that
 $\mathbf{V}_{\left(k_1^\prime,k_2^\prime\right)}\subset\mathrm{Fix\left(r,s\right)}$.
Then for any $(k_1,k_2)$ we have
$ \left(r,s\right)\cdot\left(k_1,k_2\right)=
 \mathrm{det}
\begin{bmatrix}
k_1&k_2\\
k_1^\prime&k_2^\prime
\end{bmatrix}\pmod{N}.$
Thus
$\mathbf{V}_{\left(k_1,k_2\right)}\subset\mathrm{Fix\left(r,s\right)}$ if and only if
$\mathrm{det}
\begin{bmatrix}
k_1&k_2\\
k_1^\prime&k_2^\prime
\end{bmatrix}=0\pmod{N},
$
and $\left(r,s\right)\in\mathrm{ker}
\begin{bmatrix}
k_1&k_2\\
k_1^\prime&k_2^\prime
\end{bmatrix}\pmod{N}.
$
This is equivalent to having $\left(k_1,k_2\right)$ and $\left(k_1^\prime,k_2^\prime\right)$ linearly dependent over
$\mathbb{Z}_N,$ i.e. $\left(k_1^\prime,k_2^\prime\right)=m\left(k_1,k_2\right)\pmod{N}.$ Then
$\mathrm{Fix}\left(r,s\right)=\displaystyle{\Sigma_{m\in\mathbb{Z}_N}}\mathbf{V}_{m\left(N-s,r\right)}.$
Since half of the $\left(k_1,k_2\right)=m\left(N-s,r\right)$ have $k_2>k_1,$ this expression adds two times the same subspace. Then, since for $\left(k_1,k_2\right)\neq\left(0,0\right)$ we have
$\mathrm{dim}\mathbf{V}_{\left(k_1,k_2\right)}=2$
we obtain  $\mathrm{dimFix}\left(r,s\right)=1+\frac{1}{2}\left(N-1\right)\cdot2=N.$
\end{proof}

\subsection{The $\mathbb{Z}_N\times\mathbb{Z}_N-$isotypical components of  $\mathbb{R}^{2N^2}$}

So far we have obtained the distinct $\mathbb{Z}_N\times\mathbb{Z}_{N}-$invariant representations of $\mathbb{R}^{N^2},$ by considering only the subspaces corresponding to the variable $x.$
If $\mathbf{V}\subset\mathbb{R}^{N^2}$ is an irreducible subspace for the action of $\mathbb{Z}_N\times\mathbb{Z}_{N}$ then $\mathbf{V}\times\{\mathbf{0}\}$ and $\{\mathbf{0}\}\times\mathbf{V}$ are $\mathbb{Z}_N\times\mathbb{Z}_{N}-$isomorphic irreducible subspaces of $\mathbb{R}^{2N^2}$.
We will  use the notation $\mathbf{V}\oplus\mathbf{V}$ for the subspace $\mathbf{V}\times\mathbf{V}\subset \mathbb{R}^{2N^2}$.

\begin{defi}\label{defi isotyp}
Suppose a group $\Gamma$ acts on $\mathbb{R}^n$  and
let $\mathbf{V}\subset\mathbb{R}^n$ be a $\Gamma$-irreducible subspace.
The
 isotypic component of $\mathbb{R}^n$ corresponding to $\mathbf{V}$ is the sum of all $\Gamma$-irreducible subspaces that are $\Gamma$-isomorphic to $\mathbf{V}$.
 \end{defi}

Once we have the $\mathbb{Z}_N\times\mathbb{Z}_{N}-$ireducible representations on $\mathbb{R}^{N^2},$ we can calculate the isotypic components of the representation of this group on $\mathbb{R}^{2N^2}.$

\begin{table}[h]
\centering
\begin{center}
\caption{Isotypic components $\mathbf{Z_k}$ of the $\mathbb{Z}_N\times\mathbb{Z}_{N}-$action on $\mathbb{R}^{2N^2},$ where the $\mathbf{V_k}$ are defined in \eqref{irreducible_space}.
 }\label{table 2}
\end{center}
\begin{tabular}{cccccc}
\toprule
$\mathrm{dim} \left(\mathbf{Z_k}  \right)$ & $\mathbf{Z_k}$ & Restrictions\\
 & & & \\
\midrule
\centering
$2$ & $\mathbf{V}_{(0,0)}\oplus\mathbf{V}_{(0,0)}$\\
$4$ & $\mathbf{V}_{(k_1,0)}\oplus\mathbf{V}_{(k_1,0)}$ &  $1\leqslant k_1 \leqslant \left(N-1\right)/2$\\
$4$ & $\mathbf{V}_{(0,k_2)}\oplus\mathbf{V}_{(0,k_2)}$  &  $1\leqslant k_2\leqslant \left(N-1\right)/2$\\
$4$ & $\mathbf{V}_{(k_1,k_1)}\oplus\mathbf{V}_{(k_1,k_1)}$  &  $1\leqslant k_1 \leqslant \left(N-1\right)/2$\\
$4$ & $\mathbf{V}_{(k_1,k_2)}\oplus\mathbf{V}_{(k_1,k_2)}$  &  $1\leqslant  k_2<k_1 \leqslant N-1$\\
\bottomrule
\end{tabular}
\end{table}

\begin{lema}\label{isotypic lemma}
The isotypic components for the $\mathbb{Z}_N\times\mathbb{Z}_{N}-$representation in
$\mathbb{R}^{2N^2}$ are of the form $\mathbf{Z_k} =\mathbf{V_k\oplus\mathbf{V_k}}$ given in
Table~\ref{table 2}.\end{lema}
\begin{proof}
Since  $\Gamma=\mathbb{Z}_N\times\mathbb{Z}_{N}$ acts diagonally,
$\gamma(x,y)=(\gamma x,\gamma y)$ in $\mathbb{R}^{2N^2}$ for $x,y\in \mathbb{R}^{N^2}$, then
the subspaces $\mathbf{V_k}\oplus\{\mathbf{0}\}$ and $\{\mathbf{0}\}\oplus\mathbf{V_k}$ are $\Gamma-$invariant and
irreducible, by \ref{lemma_invariance}. of Proposition~\ref{PropSubspaces}.
If $\mathbf{k}\ne\mathbf{k^\prime}$
then, by~\ref{lemma_nonisomorphic}. we have that $\mathbf{V_{k^\prime}}\oplus\{\mathbf{0}\}$ and $\{\mathbf{0}\}\oplus\mathbf{V_{k^\prime}}$ are not
$\Gamma-$isomorphic to either
$\mathbf{V_k}\oplus\{\mathbf{0}\}$ or $\{\mathbf{0}\}\oplus\mathbf{V_k}$.
Therefore, the only isomorphic representations are $\mathbf{V_k}\oplus\{\mathbf{0}\}$ and $\{\mathbf{0}\}\oplus\mathbf{V_k}$ for the same $\mathbf{k}\in I$ and the result follows.
\end{proof}

\section{Symmetries of generic oscillations patterns}\label{SecSubgroups}
The main goal of this section is to characterise the symmetries of periodic solutions of  \eqref{nFHN_2d}, specially those that
 arise at  Hopf bifurcations.

Given a solution $z(t)$ with period $P$ of a $\mathbb{Z}_N\times\mathbb{Z}_N$-equivariant differential equation $\dot{z}=f(z)$, a {\em  spatio-temporal symmetry} of $z(t)$ is a pair $(\sigma,\theta)$,
with $\sigma\in\mathbb{Z}_N\times\mathbb{Z}_N$ and
$\theta\in\mathbb{R}\pmod{P}\sim \mathbb{S}^1$ such that
$\sigma\cdot z(t)=z(t+\theta)$ for all $t$.
The group of spatio-temporal symmetries of $z(t)$ can be identified with a pair of subgroups $H$ and $K$ of $\mathbb{Z}_N\times\mathbb{Z}_N$ and a homomorphism $\Theta:H\rightarrow\mathbb{S}^1$ with kernel $K$, where $H$ represents the spatial parts of the spatio-temporal symmetries of $z(t)$, while
$K$ comprises the {\em spatial symmetries} of $z(t)$,  i.e. the symmetries that fix the solution pointwise. In order to get all the spatio-temporal symmetries for solutions of \eqref{nFHN_2d}, we use the following result of Filipitsky and Golubitsky:

\begin{priteo}[abelian Hopf  theorem \cite{AbelianHopf}]\label{aHFG}
In systems with abelian symmetry, generically, Hopf bifurcation at a point  $X_0$ occurs with simple eigenvalues, and there exists a unique branch of small-amplitude periodic solutions emanating from  $X_0$. Moreover, the spatio-temporal  symmetries of the bifurcating periodic solutions are
$H=\Sigma_{X_0}$ and $K=\ker_V(H)$, where $V$ is the centre subspace of the bifurcation at $X_0$ and $H$ acts $H$-simply on $V$.
\end{priteo}

Recall that a subgroup $H$ of $\mathbb{Z}_N\times\mathbb{Z}_N$ acts
{\em $H-$simply on a subspace $\mathbf{V}$ } if either $\mathbf{V}$ is the sum of two isomoprphic $H$-irreducible subspaces or $\mathbf{V}$ is $H$-irreducible but not absolutely irreducible.

\begin{prop}\label{PropaAbelianHopf}
Periodic solutions of \eqref{nFHN_2d} arising through Hopf bifurcation with simple eigenvalues at an equilibrium point $X_0$ have
the spatio-temporal symmetries of Table~\ref{tableAbelianHopf}.
\end{prop}

\begin{table}[ht]
\caption{Spatio-temporal ($H$) and spatial ($K$) symmetries of  \eqref{nFHN_2d} that may arise through Hopf bifurcation at a point $X_0$, with two-dimensional centre subspace $\mathbf{V}$.
Here $\mathbf{V_{k^\perp}}$ denotes the subspace
$\fix\left(\widetilde{\mathbb{Z}}_N\left(\mathbf{k}\right)\right)\subset\mathbb{R}^{N^2}$ as in the proof of \ref{IsotropySubgroups}. of Proposition~\ref{PropSubspaces}.
}\label{tableAbelianHopf}
\begin{tabular}{cccccc}
\toprule
$H$	&	set	&	centre 	&$K$ 		&	restrictions	&		\\
	&	containing $X_0$	&	 subspace $V$	& 		&	on $\mathbf{k}$	&		\\ \midrule
$\Gamma$	&	$\mathbf{V}_{\mathbf{0}}\oplus\mathbf{V}_{\mathbf{0}}$	&	 $V\subset\mathbf{V}_{\mathbf{0}}\oplus\mathbf{V}_{\mathbf{0}}$	&	 $H$	 &	 &			\\
$\Gamma$	&	$\mathbf{V}_{\mathbf{0}}\oplus\mathbf{V}_{\mathbf{0}}$	&	$V\subsetneqq\mathbf{V_k}\oplus\mathbf{V_k}$	&	 $\widetilde{\mathbb{Z}}_N\left(\mathbf{ k^\perp}\right)$	&	$ \mathbf{k} \ne \mathbf{0}$	&		\\
$\widetilde{\mathbb{Z}}_N\left(\mathbf{ k}\right)$	&	$\mathbf{V_{k^\perp}}\oplus\mathbf{V_{k^\perp}}\backslash \{\mathbf{0}\}$	 &	 $V=\mathbf{V}_{\mathbf{0}}\oplus\mathbf{V}_{\mathbf{0}}$	&	$H$	&	$ \mathbf{k} \ne \mathbf{0}$	&	\\
$\widetilde{\mathbb{Z}}_N\left(\mathbf{ k}\right)$	&	$\mathbf{V_{k^\perp}}\oplus\mathbf{V_{k^\perp}}\backslash \{\mathbf{\mathbf{0}}\}$	&	 $V\subsetneqq\mathbf{V_\ell}\oplus\mathbf{V_\ell}$		&	$\one$	&	$ \mathbf{k} \ne \mathbf{0}$	 & $\ell\cdot \mathbf{k}\ne \mathbf{0}$	 \\
$\widetilde{\mathbb{Z}}_N\left(\mathbf{ k}\right)$	&	$\mathbf{V_{k^\perp}}\oplus\mathbf{V_{k^\perp}}\backslash \{\mathbf{0}\}$	 &	 $V\subsetneqq\mathbf{V_\ell}\oplus\mathbf{V_\ell}$		&	$H$	&	$ \mathbf{k} \ne \mathbf{0}$	& $\ell\cdot \mathbf{k} =\mathbf{0}$	 \\
$\one$	&	$\mathbb{R}^{2N^2}\backslash\bigcup\mathbf{V_k}\oplus\mathbf{V_k}$	&	$V\subset\mathbf{V_k}\oplus\mathbf{V_k}$		 &	 $\one$		 &	 &\\
\bottomrule
\end{tabular}
\end{table}

\begin{proof}
The proof is a direct application of  Theorem~\ref{aHFG}, using the information of Section~\ref{sectionSymmetry}.
From assertions
\ref{IsotropySubgroups}. and \ref{FixedPointSubspaces}. of Proposition~\ref{PropSubspaces},
the possibilities for $H$ are $\one$,
$\widetilde{\mathbb{Z}}_N\left(r,s\right)$ and $\mathbb{Z}_N\times\mathbb{Z}_N$.
This yields the first  collumn in Table~\ref{tableAbelianHopf}.
The second collumn is obtained from the list of corresponding fixed-point subspaces.

Let $V$ be the centre subspace at $X_0$.  Since the eigenvalues are simple, $V$ is two-dimensional  and is contained in one of the isotypic components.
Then either $H$ acts on $V$ by nontrivial rotations and the action is irreducible but not absolutely irreducible, or  $H$ acts trivially on $V\subset\fix H$, hence $V$ is the sum of two $H$-irreducible components.
In any of these cases $H$ acts $H$-simply on $V$.
The possibilities, listed in Lemma~\ref{isotypic lemma}, yield the third collumn of
Table~\ref{tableAbelianHopf}.
The spatial symmetries are then obtained by checking whether $V$ meets $\fix(H)$.
\end{proof}

A useful general tool for identifying periodic solutions whose existence is not guaranteed by the Equivariant Hopf Theorem, is the $H~\mathrm{mod}~K$ Theorem \cite{GS03}. Although it has been shown in \cite{GS03}  that in general there may be periodic solutions with spatio-temporal symmetries predicted by the $H~\mathrm{mod}~K$ Theorem that cannot be obtained from Hopf bifurcation,
this is not the case here.

Hopf bifurcation with simple eigenvalues is the generic situation for systems with abelian symmetry \cite[Theorem 3.1]{AbelianHopf}.
Theorem~\ref{aHFG} is a kind of  ``converse'' to the $H~\mathrm{mod}~K$ Theorem in the case of generic vector fields with abelian symmetry:
it states  that, generically, the periodic solutions provided by the $H~\mathrm{mod}~K$ Theorem can be obtained through Hopf bifurcation. 
In the next section we will show in Theorem~\ref{ThSimpleEigenvalues} that this is indeed the case for  \eqref{nFHN_2d}: even though  \eqref{nFHN_2d} is not a generic equivariant vector field the conclusion of the ``converse theorem'' still holds in this case for generic values of the parameters
and we will obtain explicit genericity conditions.

\section{Linear Stability}\label{Stability}

In this section we study the stability of solutions of \eqref{nFHN_2d} lying in the {\em full synchrony subspace}
$\mathbf{V}_{(0,0)}\oplus\mathbf{V}_{(0,0)}\subset\mathbb{R}^{2N^2}$.
For this we choose  coordinates in $\mathbb{R}^{2N^2}$ by concatenating the
transposed collumns of the matrix $\left(x_{\alpha,\beta},y_{\alpha,\beta}\right)$, i.e.
the coordinates are $(C_1,\ldots,C_N)^T$ where
\begin{equation}\label{Cbeta}
C_\beta=\left(x_{1,\beta},y_{1,\beta},x_{2,\beta},y_{2,\beta},\ldots,x_{N,\beta},y_{N,\beta}\right)^T\ .
\end{equation}
Let $p\in\mathbf{V}_{(0,0)}\oplus\mathbf{V}_{(0,0)}$ be a point with all coordinates
$\left(x_{\alpha,\beta},y_{\alpha,\beta}\right)=\left(x_{*},y_{*}\right)$.
In these coordinates,
 the  linearization of \eqref{nFHN_2d}  around
$p$ is given by the $N\times N$ block circulant matrix $M$ given by
\begin{equation*}\label{matrixM}
    M =
        \begin{bmatrix}
            A&B&\mathbf{0}&\ldots&\mathbf{0}&\mathbf{0}\\
            \mathbf{0}&A&B&\mathbf{0}&\ldots&\mathbf{0}\\
            \mathbf{0}&\mathbf{0}&A&B&\ldots&\mathbf{0}\\
            \vdots&\vdots&\vdots&\vdots&\vdots&\vdots&\\
            B&\mathbf{0}&\ldots&\mathbf{0}&\mathbf{0}&A\\
        \end{bmatrix}
\end{equation*}
where $A$ is an $N\times N$ block circulant matrix and $B$ is an $N\times N$ block diagonal matrix given by
\begin{equation*}\label{matricesAB}
    A =
        \begin{bmatrix}
            D&E&\mathbf{0}&\ldots&\mathbf{0}&\mathbf{0}\\
            \mathbf{0}&D&E&\mathbf{0}&\ldots&\mathbf{0}\\
            \mathbf{0}&\mathbf{0}&D&E&\ldots&\mathbf{0}\\
            \vdots&\vdots&\vdots&\vdots&\vdots&\vdots&\\
            E&\mathbf{0}&\ldots&\mathbf{0}&\mathbf{0}&D\\
        \end{bmatrix}
    B =
        \begin{bmatrix}
            F&\mathbf{0}&\mathbf{0}&\ldots&\mathbf{0}&\mathbf{0}\\
            \mathbf{0}&F&\mathbf{0}&\mathbf{0}&\ldots&\mathbf{0}\\
            \mathbf{0}&\mathbf{0}&F&\mathbf{0}&\ldots&\mathbf{0}\\
            \vdots&\vdots&\vdots&\vdots&\vdots&\vdots&\\
            \mathbf{0}&\mathbf{0}&\ldots&\mathbf{0}&\mathbf{0}&F\\
            \end{bmatrix}
\end{equation*}
where the $2\times 2$ matrices
$E$ and $F$ are given by
\begin{equation*}\label{matricesCDE}
    E =
        \begin{bmatrix}
            -\gamma&0\\
            0&0\\
        \end{bmatrix}
    F =
        \begin{bmatrix}
            -\delta&0\\
            0&0\\
        \end{bmatrix}
\end{equation*}
and  $D$ is obtained from the matrix of the derivative $D(f_1,f_2)$ of \eqref{FHN} at $\left(x_{*},y_{*}\right)$
as $D=D(f_1,f_2)-E-F$.
In particular, if $p$ is the origin we have
\begin{equation*}\label{matrixD}
    D =
        \begin{bmatrix}
            d&-1\\
            b&-c\\
        \end{bmatrix}
        \qquad\mbox{with}\qquad
        d=-a+\delta+\gamma \ .
\end{equation*}

Given a vector $v\in\mathbb{C}^{k}$, we use the $N^{th}$ roots of unity  $\omega^r=\exp\left(2\pi ir/N\right)$ to define the vector
$\Omega(r,v)\in\mathbb{C}^{kN}$ as
$$
\begin{array}{l}
\Omega(r,v)=\left[v,\omega^{r} v,\omega^{2r} v,\ldots,\omega^{\left(N-1\right)r} v\right]^T,~0\leqslant r\leqslant N-1.
\end{array}
$$
The definition may be used recursively to define the vector
$\Xi(r,s,v)=\Omega\left(s,\Omega(r,v)\right)\in\mathbb{C}^{k^2}$ as
$$
    \begin{array}{l}
        \displaystyle \Xi(r,s,v)=
        \left[\Omega(r,v),\omega^{s}\Omega(r,v),\omega^{2s}\Omega(r,v),\ldots,\omega^{\left(N-1\right)s}\Omega(r,v)\right]^T .
    \end{array}
$$

\begin{priteo}
If $\lambda_{r,s}$ is an eigenvalue and $v\in\mathbb{C}^{2}$ an eigenvector of
$D+\omega^r E +\omega^sF$, and if $M$ is  the  linearization of \eqref{nFHN_2d}  around $p\in\mathbf{V}_{(0,0)}\oplus\mathbf{V}_{(0,0)}$
then
$\lambda_{r,s}$ is an eigenvalue of $M$ with corresponding eigenvector
$\Xi(r,s,v)$.
\end{priteo}

\begin{proof}
Let us first compute the eigenvalues of the matrix $A$. We have, for any $v\in\mathbb{C}^{2}$
$$
A \Omega(r,v)= \Omega\left(r,\left(D+\omega^r E\right)v\right)
$$
or, in full,
$$
\begin{array}{l}
        \begin{bmatrix}
            D&E&\mathbf{0}&\ldots&\mathbf{0}&\mathbf{0}\\
            \mathbf{0}&D&E&\mathbf{0}&\ldots&\mathbf{0}\\
            \mathbf{0}&\mathbf{0}&D&E&\ldots&\mathbf{0}\\
            \vdots&\vdots&\vdots&\vdots&\vdots&\vdots&\\
            E&\mathbf{0}&\ldots&\mathbf{0}&\mathbf{0}&D\\
        \end{bmatrix}
        \begin{bmatrix}
            v\\
            \omega^r v\\
            \omega^{2r} v\\
            \vdots\\
            \omega^{\left(N-1\right)r} v\\
            \end{bmatrix}
            =
        \begin{bmatrix}
           \left(D+\omega^r E\right) v\\
            \omega^r \left(D+\omega^r E\right)v\\
            \omega^{2r}\left(D+\omega^r E\right) v\\
            \vdots\\
            \omega^{\left(N-1\right)r} \left(D+\omega^r E\right)v\\
            \end{bmatrix}
        \end{array}
$$
so, if $\left(D+\omega^r E\right)v=\lambda_rv$ then
$A \Omega(r,v) =\lambda_r \Omega(r,v)$.

By applying the same algorithm we can calculate the eigenvalues and eigenvectors of the matrix $M.$ Given  $u\in\mathbb{C}^{2N}$
compute
$$
M \Omega(s,u)=\Omega\left(s, \left(A+\omega^s B\right)u \right)
$$
or, in full
$$
\begin{array}{l}
        \begin{bmatrix}
            A&B&\mathbf{0}&\ldots&\mathbf{0}&\mathbf{0}\\
            \mathbf{0}&A&B&\mathbf{0}&\ldots&\mathbf{0}\\
            \mathbf{0}&\mathbf{0}&A&B&\ldots&\mathbf{0}\\
            \vdots&\vdots&\vdots&\vdots&\vdots&\vdots&\\
            B&\mathbf{0}&\ldots&\mathbf{0}&\mathbf{0}&A\\
        \end{bmatrix}
        \begin{bmatrix}
            u\\
            \omega^s u\\
            \omega^{2s} u\\
            \vdots\\
            \omega^{\left(N-1\right)s} u\\
            \end{bmatrix}=
        \begin{bmatrix}
            \left(A+\omega^s B\right)u\\
            \omega^s \left(A+\omega^s B\right)u\\
            \omega^{2s} \left(A+\omega^s B\right)u\\
            \vdots\\
            \omega^{\left(N-1\right)s} \left(A+\omega^s B\right)u\\
            \end{bmatrix}
        \end{array}
$$
To complete the proof we compute
$M \Xi(r,s,v)=M\Omega\left(s,\Omega(r,v)\right)$ as
$$
M \Xi(r,s,v)=
\Omega\left(s, \left(A+\omega^s B\right)\Omega(r,v) \right)=
\Omega\left(s,A\Omega(r,v) \right)+\Omega\left(s,\omega^s B\Omega(r,v) \right).
$$
Then, since $B$ is a block diagonal matrix, then
$B \Omega(r,v)=\Omega(r,Fv)$ for any $v\in\mathbb{R}^{2}$ and we get:
$$
\begin{array}{ll}
M \Xi(r,s,v)&=
\Omega\left(s,\Omega\left(r,\left(D+\omega^r E\right)v\right) \right)+
\Omega\left(s,\Omega(r,\omega^sFv) \right)\\
&=
\Omega\left(s,\Omega\left(r,\left(D+\omega^r E+\omega^sF\right)v\right) \right)
\end{array}
$$
and thus
$$
M \Xi(r,s,v)=\Xi\left(r,s,\left(D+\omega^r E+\omega^sF\right)v\right) .
$$
It follows that if $\left(D+\omega^r E+\omega^sF\right)v=\lambda_{r,s}v$ then
$M \Xi(r,s,v) =\lambda_{r,s}  \Xi(r,s,v)$ as we had claimed.
\end{proof}

\subsection{Form of the eigenvalues}

\begin{priteo}\label{ThEigs}
If $L$ is the linearisation of \eqref{FHN} around the origin and $D=L-E-F$  then
the eigenvalues of  $D+\omega^r E +\omega^s F$ are of the form
\begin{equation}\label{eigvals}
    \begin{array}{lcl}
    	\dpt\lambda_{\left(r,s\right)\pm}
        		&=&\dpt\frac{1}{2}\left[-\left(c+a\right)+\gamma(1-\omega^r)+\delta(1-\omega^s)\right]\\
        \\
	&&\dpt\pm\frac{1}{2}\sqrt{\left[\left(c-a\right)+\gamma(1-\omega^r)+\delta(1-\omega^s)\right]^2-4b}
    \end{array}
\end{equation}
where $\sqrt{-}$ stands for the principal square root.
Moreover, on the isotypic component $\mathbf{V}_{(k_1,k_2)}\oplus\mathbf{V}_{(k_1,k_2)}$,
the eigenvalues of $M$ are
$\lambda_{\left(k_1,k_2\right)\pm}$ and their complex conjugates
$\lambda_{\left(N-k_1,N-k_2\right)\pm}$.
\end{priteo}

\begin{proof}
It is straightforward to derive the explicit expression \eqref{eigvals} of the  eigenvalues of
$D+\omega^r E +\omega^sF$; a direct calculation shows that, unless $(r,s)=(0,0)$,
 the complex conjugate of
$\lambda_{\left(r,s\right)+}$ is not  $\lambda_{\left(r,s\right)-}$,
but rather $\lambda_{\left(N-r,N-s\right)+}$.

We claim that for any complex number $\zeta$
the real and imaginary parts of  $\Xi(k_1,k_2,\zeta)$ lie in $\mathbf{V}_{(k_1,k_2)}$;
from this it follows that the real and imaginary parts of the eigenvectors $\Xi(k_1,k_2,v)$ lie in the isotypic component $\mathbf{V}_{(k_1,k_2)}\oplus\mathbf{V}_{(k_1,k_2)}$.
Since $\mathbf{V}_{(k_1,k_2)}=\mathbf{V}_{(N-k_1,N-k_2)}$, this will complete the proof
that the eigenvalues $\lambda_{\left(k_1,k_2\right)\pm}$ and $\lambda_{\left(N-r,N-s\right)\pm}$ correspond to
$\mathbf{V}_{(k_1,k_2)}\oplus\mathbf{V}_{(k_1,k_2)}$.

It remains to establish our claim.
Using the expression \eqref{Cbeta} to write the coordinate $x_{\alpha,\beta}$ of $\Xi(k_1,k_2,\zeta)$
we obtain
$$
\begin{array}{lcl}
x_{\alpha,\beta}&=&\displaystyle
\zeta\exp\left[\frac{2\pi i}{N}\left(\alpha-1\right)k_1+\left(\beta-1\right)k_2\right]=\\
&=&\displaystyle
\zeta\exp\left[\frac{2\pi i}{N}\left(-k_1-k_2\right)\right]
\exp\left[\frac{2\pi i}{N}\left(\alpha,\beta\right)\cdot\mathbf{k}\right]\ .
\end{array}
$$
Its real and imaginary parts are of the form \eqref{irreducible_space} for
$z=\zeta\exp\left[\frac{2\pi i}{N}\left(-k_1-k_2\right)\right]$
and
$z=-i\zeta\exp\left[\frac{2\pi i}{N}\left(-k_1-k_2\right)\right]$, respectively,
and therefore lie in $\mathbf{V}_{(k_1,k_2)}\oplus\mathbf{V}_{(k_1,k_2)}$, as claimed.
\end{proof}

For  the mode $(r,s)=(0,0)$ the expression \eqref{eigvals} reduces to the eigenvalues
$\lambda_{\pm}$ of uncoupled equations \eqref{FHN}, linearised about the origin,
\begin{equation}\label{eigvalsFHN1}
    \begin{array}{l}
        \displaystyle{\lambda_{\pm}=\frac{-c-a\pm\sqrt{(c+a)^2-4b}}{2}} \ .
    \end{array}
\end{equation}

\section{Bifurcation for  $c=0$}\label{secBifciszero}
In this section we look at the Hopf bifurcation in the case $c=0$,
regarding $a$ as a bifurcation parameter.
The bulk of the section consists of the proof of  Theorem~\ref{ThFirstHopf} below.
Since  in this case the only equilibrium is the origin, only the first two rows of Table~\ref{tableAbelianHopf} occur.

\begin{priteo}\label{ThSimpleEigenvalues}
For generic $\gamma,\delta$ and for $c=0,~b\ne 0$ all
 the eigenvalues of  the  linearization of  \eqref{nFHN_2d}  around the origin have multiplicity $1$.
\end{priteo}

\begin{proof}
We can write the characteristic polynomial for $L+(\omega^r-1) E +(\omega^s-1)F$,
where $L$ is the linearization of \eqref{FHN} about the origin, as
\begin{equation}\label{multiplicity_f}
        f(\lambda,r,s)=\   \lambda^2+\lambda\left[a-\gamma\left(1-\omega^{r}\right) -\delta\left(1-\omega^{s}\right)       \right]+b     \qquad 1\leqslant r,s\leqslant N.
\end{equation}
We start by showing that if two of these polynomials have one root in common, then they  are identical.

Indeed, let $\phi_1(\lambda)$ and $\phi_2(\lambda)$ be two polynomials of the form \eqref{multiplicity_f} and suppose they share one root, say $\lambda=p+iq,$ while the remaining roots are $\lambda=p_1+iq_1$ for $\phi_1$ and $\lambda=p_2+iq_2$ for $\phi_2$.
Since $b\ne 0$, then none of these roots is zero.
Then we can write
$$
\begin{array}{ll}
\phi_j(\lambda)&=\left(\lambda-\left(p+iq\right)\right) \left(\lambda-\left(p_j+iq_j\right)\right)\\
&=
\lambda^2-\lambda\left(\left(p+p_j\right)+i\left(q+q_j\right)\right)+\left(p+iq\right)\left(p_j+iq_j\right)
\end{array}
$$
and therefore
 $\left(p+iq\right)\left(p_1+iq_1\right)=b =\left(p+iq\right)\left(p_2+iq_2\right)$,
so, as $p+iq\ne 0$, then  $\left(p_1+iq_1\right)=\left(p_2+iq_2\right)$ and therefore
$\phi_1(\lambda)=\phi_2(\lambda)$.
Since this is valid for any pair of polynomials of the family, it only remains to show that for generic
$\gamma,\delta$ the polynomials do not coincide.

Two polynomials $f(\lambda,r,s)$ and  $f(\lambda,\tilde{r},\tilde{s})$   of the form \eqref{multiplicity_f} coincide if and only if
\begin{equation}\label{condition}
\gamma\left(\omega^{r}-\omega^{\tilde{r}}\right) =
\delta\left(\omega^{\tilde{s}}-\omega^{s}\right).
\end{equation}
Thus, for $(\gamma,\delta)$ outside a finite number of lines defined by \eqref{condition}
all the eigenvalues of  the  linearization of \eqref{nFHN_2d}  around the origin have multiplicity $1$, as we wanted to show.
\end{proof}

\begin{priteo}\label{ThFirstHopf}
For $c=0$, $b>0$ and for any $\gamma$ and $\delta$ with $\gamma\delta\neq 0$
 the origin is the only equilibrium of  \eqref{nFHN_2d}.
 For each value of $\gamma$ and $\delta$ there exists $a_*\geqslant 0$ such that
 for $a\geqslant a_*$
 the origin is asymptotically stable.
 The stability of the origin changes at $a=a_*$, where it undergoes a Hopf bifurcation
 with respect to the bifurcation parameter $a$, into a 
 periodic solution.
The spatial symmetries of the bifurcating solution and the values of $a_*$ are given in Table~\ref{tabelaBifurca}.
Moreover, if  the coupling is associative, i.e, if both $\gamma<0$ and $\delta<0$,  the bifurcating solution is stable and the bifurcation is subcritical.
\end{priteo}

\begin{table}[h]
\caption{Details of Hopf bifurcation on the parameter $a$ for Theorem~\ref{ThFirstHopf}.
Solutions bifurcate at $a=a_*$ (where $\theta_N=\frac{(N-1)\pi}{N}$) with spatial symmetry $K$ and spatio-temporal symmetry
$\mathbb{Z}_N\times\mathbb{Z}_{N}$.}\label{tabelaBifurca}
\begin{center}\begin{tabular}{cccc}
\toprule
sign$(\gamma)$& sign$(\delta)$& $a_*$ & $K$\\
\midrule
-&-&0&$\mathbb{Z}_N\times\mathbb{Z}_{N}$\\
+&-&$\gamma\left(1-\cos\theta_N\right)$&$\widetilde{\mathbb{Z}}_N\left(0,1\right)$\\
-&+&$\delta\left(1-\cos\theta_N\right)$&$\widetilde{\mathbb{Z}}_N\left(1,0\right)$\\
+&+&$(\gamma+\delta)\left(1-\cos\theta_N\right)$&$\widetilde{\mathbb{Z}}_N\left(\frac{N-1}{2},\frac{N-1}{2}\right)$\\
\bottomrule
\end{tabular}\end{center}
\end{table}

The first step is to determine the stability of the origin. To do this, we need estimates for the real part  of the eigenvalues \eqref{eigvals}. This is done in the next Lemma.

\begin{lema}\label{LemaRealPart}
Let $A(r,s)=-a+\gamma\left(1-\omega^r\right)+ \delta\left(1-\omega^s\right)$.
For $c=0$, $b>0$, $\gamma\delta\neq 0$
and for all $(r,s)$
 we have
$\mathrm{Re}\ \lambda_{\left(r,s\right)-}\leqslant \frac{1}{2} \mathrm{Re}\ A(r,s)$.
If $\mathrm{Re}\ A(r,s)\geqslant 0$ then
$\mathrm{Re}\ \lambda_{\left(r,s\right)+}\leqslant  \mathrm{Re}\ A(r,s)$,
otherwise $\mathrm{Re}\ \lambda_{\left(r,s\right)+}<0$.
\end{lema}

\begin{proof}[Proof of Lemma~\ref{LemaRealPart}]
In order to evaluate the  real and imaginary parts of eigenvalues $\lambda_{\left(r,s\right)\pm}$,
we need to rewrite equation \eqref{eigvals} by getting rid of the square root.
For this purpose, we use a well known result from elementary algebra; we have that if
$\eta=a_1+ib_1$, where $a_1$ and $b_1$ are real, $b_1\ne 0$, then the real and imaginary parts of
$\sqrt{\eta}=\sqrt{a_1+ib_1}$ are given by
\begin{equation}\label{RaizEta}
	\mathrm{Re}\ \sqrt{\eta}
        =\sqrt{\frac{\left|\eta\right|+a_1}{2}}
        \qquad
        \mathrm{Im} \ \sqrt{\eta}
       ={\mathrm{sgn}\left(b_1\right)}\sqrt{\frac{\left|\eta\right|-a_1}{2}}.
\end{equation}
A direct application of \eqref{eigvals} in Theorem~\ref{ThEigs} to the case $c=0$ yields, if
$A^2\not\in\mathbb{R}$ and for $\varepsilon_1=\pm 1$
\begin{equation}\label{Real}
\mathrm{Re}\ \lambda_{\left(r,s\right)\varepsilon_1}=
\frac{1}{2}\left(\mathrm{Re}\ A+\varepsilon_1 \sqrt{\frac{\left|A^2-4b \right| +\mathrm{Re}\ (A^2-4b)} {2}}\right)
\end{equation}
\begin{equation}\label{Imaginary}
\mathrm{Im}\ \lambda_{\left(r,s\right)\varepsilon_1}=
\frac{1}{2}\left(\mathrm{Im}\ A+\varepsilon_1{\mathcal {S}} \sqrt{\frac{\left|A^2-4b \right| -\mathrm{Re}\ (A^2-4b)} {2}}\right)
\end{equation}
with $A=A(r,s)$
and ${\mathcal {S}} = \mathrm{sgn} \left(\mathrm{Im}\left(A^2\right)\right)$.

The statement for $\lambda_{\left(r,s\right)-}$ follows immediately from \eqref{Real}.
For $\lambda_{\left(r,s\right)+}$, note that, for $b>0$ and any $\eta\in\mathbb{C}$, we have
$$
\left|\eta^2-4b \right| +\mathrm{Re}\ (\eta^2-4b)\leqslant 2\left(\mathrm{Re}\ \eta\right)^2
$$
with equality holding if and only if  $\mathrm{Re}\ \eta=0$, when the expressions are identically zero.
Hence, taking $\eta=A(r,s)$, we obtain from \eqref{Real}:
$$
\mathrm{Re}\ \lambda_{\left(r,s\right)+}\leqslant \frac{1}{2}
\left( \mathrm{Re}\ A(r,s)+\left| \mathrm{Re}\ A(r,s)\right|\right)
$$
and the result follows.
\end{proof}

The particular case of fully synchronised solutions in Theorem~\ref{ThFirstHopf} is treated in the next Lemma. 
This case is simpler since the bifurcation takes place inside the subspace
$\mathbf{V}_{(0,0)}\oplus\mathbf{V}_{(0,0)}$.

\begin{lema}\label{HopfSynchro}
For $c=0$, $b>0$ and $\gamma\delta\neq 0$  the origin is the only equilibrium of  \eqref{nFHN_2d} and  at $a=0$ it undergoes a Hopf bifurcation,  subcritical with respect to the bifurcation parameter $a$, to a fully synchronised periodic solution.
If both $\gamma<0$ and $\delta<0$, the origin is asymptotically stable for $a>0$ and the bifurcating solution is stable.
Otherwise, the periodic solution is unstable.
\end{lema}

\begin{proof}[Proof of Lemma~\ref{HopfSynchro}]
Inspection of \eqref{nFHN_2d} when $c=0$ shows that the only equilibrium is the origin.

The restriction of  \eqref{nFHN_2d} to  the plane
$\mathbf{V}_{(0,0)}\oplus\mathbf{V}_{(0,0)}$
obeys the  uncoupled equations \eqref{FHN} whose  linearisation  around the origin has eigenvalues given by  \eqref{eigvalsFHN1}.
It follows that, within $\mathbf{V}_{(0,0)}\oplus\mathbf{V}_{(0,0)}$, the origin is asymptotically stable for $a>0$, unstable for $a<0$. The linearisation has purely imaginary eigenvalues at $a=0$.

Consider the positive function $\varphi(y)=\exp(-2y/b)$. Then, for  the  uncoupled equations \eqref{FHN} we get
$$
\frac{\partial}{\partial x}\left(\varphi(y)\dot{x}\right)+\frac{\partial}{\partial y}\left(\varphi(y)\dot{y}\right)=
\left(-3x^2 +2ax -a \right)\varphi(y)
$$
which is always negative if  $0\le a\le 3$.
Hence,  by Dulac's criterion, the system  \eqref{FHN}  cannot have any periodic solutions,
and thus  if there is a Hopf bifurcation at $a=0$ inside  the plane $\mathbf{V}_{(0,0)}\oplus\mathbf{V}_{(0,0)}$ it  must be subcritical.

In order to show that indeed there is  a Hopf bifurcation we apply the criterium of \cite[Theorem 3.4.2]{guck} and evaluate
$$
\begin{array}{ll}
        16s^*=&f_{xxx}+f_{xyy}+g_{xxy}+g_{yyy}\\
        \\
       &\displaystyle{+\frac{1}{\sqrt{b}}\left[f_{xy}\left(f_{xx}+f_{yy}\right)-g_{xy}\left(g_{xx}+g_{yy}\right)-f_{xx}g_{xx}+f_{yy}g_{yy}\right]},
\end{array}
$$
where $f(x,y)=f_1(x,\sqrt{b}y)-\sqrt{b}y$, $g(x,y)=f_2(\sqrt{b}x,y)-\sqrt{b}x$ and
$f_{xy}$ denotes $\displaystyle\frac{\partial^2 f}{\partial x\partial y}(0,0)$, etc.
Since, for $a=c=0$, we have $f(x,y)=-x^3+x^2$ and $g(x,y)=0$, this yields
$s^*=-\frac{3}{8}$. The Hopf bifurcation is not degenerate and the bifurcating  periodic solution  is stable within $\mathbf{V}_{(0,0)}\oplus\mathbf{V}_{(0,0)}$.
Since $\partial \mathrm{Re}\ \lambda_{\left(0,0\right)\pm}/\partial a=-1/2$ the bifurcation is indeed subcritical with respect to the bifurcation parameter $a$.

It remains to discuss the global stability, with respect to initial conditions outside
 $\mathbf{V}_{(0,0)}\oplus\mathbf{V}_{(0,0)}$.
If $\gamma>0$ then from the expression \eqref{Real} in the proof of Lemma~\ref{LemaRealPart} at $a=0$ we obtain
$\mathrm{Re}\ \lambda_{\left(1,0\right)+}>0$ and the bifurcating periodic solution is unstable. A similar argument holds for $\delta>0$.

If both $\gamma<0$ and $\delta<0$, then, for  $a=0$, we get
$\mathrm{Re}\ A(r,s)<0$  for $(r,s)\neq (0,0)$.
Hence, by Lemma~\ref{LemaRealPart}, all the eigenvalues $\lambda_{\left(r,s\right)\pm}$,  $(r,s)\neq (0,0)$ have negative real parts and the bifurcating solution is stable.
\end{proof}
\medbreak

\begin{proof}[Proof of Theorem~\ref{ThFirstHopf}]
The case  of associative coupling $\gamma<0$ and $\delta<0$ having been treated in Lemma~\ref{HopfSynchro},
it remains to deal with the cases when either $\gamma$ or $\delta$ is positive.
Let $r_{\varepsilon_2}=\frac{N+\varepsilon_2}{2}$, $\varepsilon_2=\pm1$, $c=0$, $b>0$, 
$\theta_N=\frac{\pi(N-1)}{N}$, with $\sin\theta_N>0$, $\cos\theta_N<0$ and 
$\cos\theta_N<\cos\frac{ 2\alpha\pi}{N}$ for all $\alpha\in\mathbb{Z}$.

If $\gamma>0$, $\delta<0$ and
$a\geqslant \gamma\left(1-\cos\theta_N\right)=a_*$,
then for all $(r,s)\neq \left(r_{\varepsilon_2},0\right)$ we have
$\mathrm{Re}\ A(r,s)\leqslant 0$, with equality only if both $a=a_*$ and
$(r,s)=(r_{\varepsilon_2},0)$. Using \eqref{Real} and \eqref{Imaginary} we get
$$
A(r_{\varepsilon_2},0)=i\varepsilon_2\gamma\sin\theta_N;
$$
$$
\mathrm{Im}\lambda_{(r_{\varepsilon_2},0)\varepsilon_1}=
\frac{1}{2}\left(\varepsilon_2\gamma\sin\theta_N+
\varepsilon_1\sqrt{\gamma^2\sin^2\theta_N+4b}\right)
$$
with  $\lambda_{(r_+,0)+}=\overline{\lambda_{(r_-,0)-}}$ and
$\lambda_{(r_-,0)+}=\overline{\lambda_{(r_+,0)-}}$. In addition
$\mathrm{Im}\lambda_{(r_+,0)+}>\mathrm{Im}\lambda_{(r_-,0)+}>0$ for $b>0$.

The results of Golubitsky and Langford \cite{GoLangf} are always applicable to the Hopf bifurcation for $\lambda_{(r_+,0)+}$, since there are no eigenvalues of the form $k\lambda_{(r_+,0)+}$ with $k\in \mathbb{N}$.
For the smaller imaginary part there may be resonances when $\lambda_{(r_+,0)+}=k\lambda_{(r_-,0)-}$ with $k\in \mathbb{N}$.
Otherwise, if the other non-degeneracy conditions hold, there are two independent Hopf bifurcations at $a=a_*$ i.e. two separate solution branches that bifurcate at this point.
The resonance condition may be rewritten as
$$
\gamma^2{\sin\theta_N}=\frac{(k-1)^2}{k}{b},\quad k\in\mathbb{N},~k\geqslant 2.
$$
The bifurcating solutions are stable if and only if the branches are subcritical.
The eigenspace corresponding to these branches lies in $V_{(r_{\pm},0)}\oplus V_{(r_{\pm},0)}\subset\mathrm{Fix}(\mathbb{\tilde{Z}}(0,1)).$

In the case $\gamma>0$, $\delta>0$, we have
$a_*=(\gamma+\delta)\left(1-\cos\theta_N\right)$.
For $a\geqslant a_*$ and for all $(r,s)$ we have
$\mathrm{Re}\ A(r,s)\leqslant 0$, and hence $\mathrm{Re}\ \lambda_{(r,s)\pm}\leqslant 0$
with equality holding only when both  $a=a_*$ and  $\lambda_{(r_{\pm},r_{\pm})+}$.
The eigenspace in this case lies in
$\mathbf{V}_{(r_{\pm},r_{\pm})}\oplus\mathbf{V}_{(r_{\pm},r_{\pm})}\subset
\fix\left(\widetilde{\mathbb{Z}}_N\left(\frac{N-1}{2},\frac{N-1}{2}\right)\right)$.
Then
$$
A(r_{\varepsilon_2},r_{\varepsilon_3})=
i(\varepsilon_2\gamma+\varepsilon_3\delta)\sin\theta_N
$$
$$
\lambda_{(r_{\varepsilon_2},r_{\varepsilon_3})_{\varepsilon_1}}=\frac{1}{2}\left(A(r_{\varepsilon_2},r_{\varepsilon_3})+
\varepsilon_1\sqrt{A^2(r_{\varepsilon_2},r_{\varepsilon_3})-4b}\right),
$$
with $\varepsilon_i=\pm1, i=\{1,2,3\}.$
Hence,
$\mathrm{Re}\lambda_{(r_{\varepsilon_2},r_{\varepsilon_3})\varepsilon_1}=0$ and
$$
\mathrm{Im}\lambda_{(r_{\varepsilon_2},r_{\varepsilon_3})\varepsilon_1}=
\frac{1}{2}\left((\varepsilon_2\gamma+\varepsilon_3\delta)\sin\theta_N+
\varepsilon_1\sqrt{(\varepsilon_2\gamma+\varepsilon_3\delta)^2\sin^2\theta_N+4b}\right)
$$
with 
$\lambda_{(r_{-\varepsilon_2},r_{-\varepsilon_3})-}
=\overline{\lambda_{(r_{\varepsilon_2},r_{\varepsilon_3)+}}}$.
 In addition
$\mathrm{Im}\lambda_{(r_{\varepsilon_2},r_{\varepsilon_3})\varepsilon_1}>0$ when
$\varepsilon_1=+1$ with $\mathrm{Im}\lambda_{(r_+,r_+)+}>
\mathrm{Im}\lambda_{(r_{\varepsilon_2},r_{\varepsilon_3})+}$ for 
$({\varepsilon_2},{\varepsilon_3})\ne (+1,+1)$ and  $b>0$.

Hence there is a non-resonant Hopf bifurcation corresponding to  $\lambda_{(r_+,r_+)+}$.
As  mentioned before, for the smaller imaginary parts there may be resonances when 
$\lambda_{(r_{\varepsilon_2},r_{\varepsilon_3})+}=k\lambda_{(r_+,r_+)+}$ with $k\in \mathbb{N}.$
Otherwise, there are four independent Hopf bifurcations at $a=a_*$ if other non-degeneracy conditions hold; in this case four separate solution branches  bifurcate at this point. The bifurcating solutions are stable if and only if the branches are subcritical.
\end{proof}

We have checked numerically the  non-degeneracy condition for bifurcation of the non-resonant branch using the  
formulas  of  Golubitsky and Langford \cite{GoLangf}. 
 The criticality of the bifurcation branch seems to depend on $N$. 
 For $\gamma>0$, $\delta>0$, $N\geqslant 11$, the bifurcating solution branch seems to be always subcritical, and hence stable. For $N=3,5,7$ it seems to be supercritical.
If $\delta\leqslant 0$, and $N\geqslant 11$,  the bifurcating branch seems to be subcritical for large  values of $\gamma>0$, supercritical otherwise.

\section{Bifurcation for $c>0$ small}\label{seccsmall}

In this section we extend the result of section~\ref{secBifciszero} for bifiurcations at small positive values of $c$.
We start with the case when both $\gamma$ and $\delta$ are negative.

\begin{coro}\label{CorolFirstHopf}
For small values of $c$, if $b>0$,  $\gamma<0$ and $\delta<0$ the origin is an equilibrium of  \eqref{nFHN_2d} and there is a neighbourhood of the origin containing no other equilibria.
There exists $\hat{a}$ near 0 such that for $a\geqslant \hat{a}$
 the origin is asymptotically stable.
 The stability of the origin changes at $a=\hat{a}$, where  it undergoes a Hopf bifurcation,  subcritical with respect to the bifurcation parameter $a$, into a stable  periodic solution
with  spatial symmetries $\mathbb{Z}_N\times\mathbb{Z}_{N}$.
\end{coro}

\begin{proof}
The eigenvalues of $Df(\mathbf{0})$ are all non-zero at $c=0$, as shown in the proof of Lemma~\ref{HopfSynchro}.
Hence $Df(\mathbf{0})$ is non-singular and the implicit function theorem
ensures that,  for small values of $c$, there is a unique  equilibrium close to the origin.
From the symmetry it follows that this equilibrium is the origin.

If both $\gamma<0$ and $\delta<0$, then it follows from the proof of Lemma~\ref{HopfSynchro} that 
for $c=0$ the purely imaginary eigenvalues at $a=a_*=0$ are simple. 
Continuity of the eigenvalues ensures the persistence of the purely imaginary pair for $c\ne 0$ at nearby values of $a$ .
Hence the corresponding eigenvectors depend smoothly on $c$,
and the non-degeneracy conditions persist for small values of $c$.
\end{proof}

The cases when either $\gamma>0$ or $\delta>0$, are treated in the next proposition.
\begin{prop}\label{propCpositive}
For small values of $c>0$, if $b>0$ and $\gamma\delta\neq 0$
 the origin is an equilibrium of  \eqref{nFHN_2d} and there is a neighbourhood of the origin containing no other equilibria.
For $a>a_*$, where $a_*$ has the value of Table~\ref{tabelaBifurca}, the origin is asymptotically stable.
 For  almost all values of  $\gamma\neq 0 $ and $\delta\neq 0$,
 The stability of the origin changes at $a=\hat{a}<a_*$, with $\hat{a}$, near $a_*$, where  it undergoes a non-resonant Hopf bifurcation into a  periodic solution
having the spatial symmetries of Table~\ref{tabelaBifurca}.
If  the bifurcation is subcritical with respect to the bifurcation parameter $a$, then the bifurcating periodic solution is stable.
\end{prop}

\begin{proof}
For small values of $c$, the origin is locally the only equilibrium, by the arguments given in the proof of 
Corollary~\ref {CorolFirstHopf}.
In Lemma~\ref{lemaStabilityOfOrigin} below, we show that for small $c>0$ and  $a\ge a_*$, all the eigenvalues of the linearisation have negative real parts.
Hence, the origin is asymptotically stable.
When $a$ decreases from $a_*$ the real parts of some eigenvalues change their signs.
It was shown in the proof of Theorem~\ref{ThFirstHopf} that for $c=0$, there are several pairs of purely imaginary eigenvalues at $a=a_*$.
In  Lemmas~\ref{LemaDeltaNeg} and \ref{LemaDeltaPos} below, we show that, generically, for  small  $c>0$, when $a$ decreases from the value $a_*$ of Table~\ref{tabelaBifurca}, the first bifurcation at 
$a=\hat{a}<a_*$ takes place when a single pair of eigenvalues crosses the imaginary axis at a non-resonant Hopf bifurcation. We also identify the pair of eigenvalues for which the first bifurcation takes place.
\end{proof}

\begin{lema}\label{lemaStabilityOfOrigin}
For small values of $c>0$, if $b>0$, $\gamma\delta\neq 0$, let $a_*$ have the value of Table~\ref{tabelaBifurca}.
If $a\geqslant a_*$ then, for all $r,s$,  and for$\varepsilon_1=\pm 1$, we have  $\mathrm{Re}\lambda_{(r,s)\varepsilon_1}<0$, and the origin is asymptotically stable.
\end{lema}

\begin{proof}
For $\varepsilon_1=\pm 1$, the eigenvalues $\lambda_{(r,s)\epsilon_1}$ have the form
\begin{equation}\label{eq 1 small c positive}
2\lambda_{(r,s)\epsilon_1}=A(r,s)-c+\epsilon_1\sqrt{\left(A(r,s)+c\right)^2-4b}.
\end{equation}
Using \eqref{RaizEta}
and writing $A(r,s)=x+iy$, we have
\begin{equation}\label{eq 2 small c positive}
(A+c)^2=(c+x)^2-y^2+2i(cy+xy).
\end{equation}
Then $\mathrm{Re}\lambda_{(r,s)-}\leqslant \mathrm{Re}\lambda_{(r,s)+}$, 
and $\mathrm{Re}\lambda_{(r,s)+}\leqslant 0$
if and only if
$$
\mathrm{Re}\sqrt{\left(A(r,s)+c\right)^2-4b}\leqslant(x+c)
$$
This never happens if $c-x<0$, and in this case we also have $2\mathrm{Re}\lambda_{(r,s)-}\leqslant 0$. 
If $c-x\geqslant 0$ let
$$
p_1=(c-x)^2-4xc+y^2+4b \quad\mbox{and}\quad
p_2=\left[(c+x)^2-y^2-4b\right]^2+4y^2(c+x)^2
$$
with $p_1^2-p_2=16\left[-x^3c+(b+2c^2)x^2-(2cb+cy^2+c^3)x+c^2b\right]$.

With this notation, $\mathrm{Re}\lambda_{(r,s)+}\leqslant 0$ if and only if $p_1>0$ and
$$
\left[(c+x)^2-y^2-4b\right]^2+4y^2(c+x)^2\leqslant\left[(c-x)^2-4xc+y^2+4b\right]^2.
$$ 
We have the following cases:
\begin{enumerate}
\item 
if $x<0$, $c>0$ then $p_1>0$ and $p_1^2-p_2>0$ and so $\mathrm{Re}\lambda_{(r,s)^+}<0$;
\item 
if $x=0$, $c>0$ then $p_1=y^2+4b+c^2>0$ and $\displaystyle{\frac{p_1^2-p_2}{16}}=c^2b>0$ and therefore $\mathrm{Re}\lambda_{(r,s)^+}<0;$
\item 
at $x=c$ we have $p_1^2-p_2=-16c^2y^2<0$,  so $(p_1^2-p_2)$ changes sign for some $x_*$,  $0<x_*<c$.
\end{enumerate}
This completes the proof, since for $a\geqslant a_*$ we have  $x=\mathrm{Re}A(r,s)\leqslant 0$, as in the proof of Theorem~\ref{ThFirstHopf}.
\end{proof}

\begin{lema}\label{LemaDeltaNeg}
For small values of $c>0$, if $b>0$, $\gamma>0$,
$\delta<0$ and if $a-a*<0$ is small, then all the eigenvalues of the linearisation of  \eqref{nFHN_2d} around the origin have real parts smaller than the real part of $\lambda_{(r_+,0)+}$
\end{lema}

\begin{proof}
It was shown in the proof of Theorem~\ref{ThFirstHopf} that at $c=0$, $a=a_*$, the eigenvalues 
$\lambda_{(r_{\varepsilon_2},0)\varepsilon_1}$, with $\varepsilon_1=\pm 1$with $\varepsilon_2=\pm 1$, are purely imaginary, and all other eigenvalues have negative real parts.
From the expression \eqref{eq 1 small c positive} it follows that 
$$
2\mathrm{Re}\lambda_{(r_{\varepsilon_2},0)\epsilon_1}=
\mathrm{Re}A(r_{\varepsilon_2},0)-c+\epsilon_1\mathrm{Re}\sqrt{\left(A(r_{\varepsilon_2},0)+c\right)^2-4b}
$$
hence  $\mathrm{Re}\lambda_{(r_{\varepsilon_2},0){-}}<\mathrm{Re}\lambda_{(r_{\varepsilon_2},0){+}}$.
Let $\hat{a}$ be the value of $a$ for which the pair 
$\lambda_{(r_{+},0){+}}=\overline{\lambda_{(r_{-},0){+}}}$ first crosses the imaginary axis.
Since $\mathrm{Re}A(r_{\varepsilon_2},0)$ decreases with $a$, then $\hat{a}<a_*$.
The estimates above show that at $a=\hat{a}$, the second pair $\lambda_{(r_{+},0){-}}=\overline{\lambda_{(r_{-},0){-}}}$ still has negative real part.
For small $c$, the other eigenvalues still have negative real parts at $\hat{a}$, by continuity.
\end{proof}

\begin{lema}\label{LemaDeltaPos}
For small values of $c>0$, if $b>0$, $\gamma>0$,
$\delta>0$ and if $a-a*<0$ is small, then all the eigenvalues of the linearisation of  \eqref{nFHN_2d} around the origin have real parts smaller than the real part of $\lambda_{(r_+,r_+)+}$
\end{lema}

\begin{proof}
As in Lemma~\ref{LemaDeltaNeg} we use \eqref{eq 1 small c positive} to show that
$\mathrm{Re}\lambda_{(r_{\varepsilon_2},r_{\varepsilon_3}){-}}<\mathrm{Re}\lambda_{(r_{\varepsilon_2},r_{\varepsilon_3}){+}}$
at $c>0$. 
It remains to compare the real parts of the two pairs
$\lambda_{(r_+,r_+)+}=\overline{\lambda_{(r_-,r_-)+}}$ 
and $\lambda_{(r_+,r_-)+}=\overline{\lambda_{(r_-,r_+)+}}$.
To do this, we write $A(r_{\varepsilon_2},r_{\varepsilon_3})=
x+iy$, where $x,y\in\RR$ and obtain conditions on $x$ and $y$ ensuring that the eigenvalue is purely imaginary.
Then we evaluate these conditions on the expressions for $x$ and $y$ to obtain the result.

From \eqref{eq 1 small c positive} and \eqref{eq 2 small c positive} we get that $\mathrm{Re}\lambda=0$ 
if and only if
$$
\left( 2(c-x)^2-X\right)^2=X^2+4(x+c)^2y^2
\qquad\mbox{for}\qquad
X=(x+c)^2-y^2-4b
$$
and this is equivalent to 
$$
(c-x)^2=X+\frac{(c+x)^2}{(c-x)^2}y^2 =(x+c)^2-y^2-4b+\frac{(c+x)^2}{(c-x)^2}y^2 
$$
which may be rewritten as:
$$
(c-x)^2-(x+c)^2+4b=y^2 \left[\frac{(c+x)^2}{(c-x)^2}-1 \right] .
$$
This may be solved for $y^2$ to yield
\begin{equation}\label{y2PsiDex}
y^2=\frac{(b-cx)(c-x)^2}{cx}=\left(\frac{b}{cx}-1\right)(c-x)^2=\psi(x)
\end{equation}
and note that, for $c>0$ and if $c^2<b$, then $\psi(x)>0$ for $0<x<c$, and in this interval $\psi(x)$ is monotonically decreasing.

Now consider the expressions of the real and imaginary parts of  
$A(r_{\varepsilon_2},r_{\varepsilon_3})$. The real part
$x(r_{\varepsilon_2},r_{\varepsilon_3})$ satisfies
$$
x(r_{\varepsilon_2},r_{\varepsilon_3})=-a+(\gamma+\delta)\left(1-\cos \frac{(N-1)\pi}{N}\right)=-a+a_* 
$$
hence, $x$ does not depend on $\varepsilon_2$ nor on $\varepsilon_3$, and  $x>0$ for $a<a_*$.

On the other hand,
the imaginary part $y(r_{\varepsilon_2},r_{\varepsilon_3})$ is
$$
y(r_{\varepsilon_2},r_{\varepsilon_3})=\left(\varepsilon_2\gamma+\varepsilon_3\delta\right)\sin\frac{(N-1)\pi}{N} 
$$
thus $y$ does not depend on $a$, and since  $\sin\frac{(N-1)\pi}{N}>0$ then,
$$
\left| y(r_{+},r_{-})\right|< y(r_{+},r_{+})  \qquad\mbox{and}\qquad  y(r_{+},r_{+}) >0
.
$$
 
Finally, when  $a$ decreases from $a_*$, then  $x$ increases from zero, and hence $\psi(x)$ decreases from $+\infty$. The first value of $y$ to satisfy \eqref{y2PsiDex} will be $y(r_{+},r_{+}) $ since it has the largest absolute value. Hence the first pair of eigenvalues  to cross the imaginary axis will be 
$\lambda_{(r_+,r_+)\varepsilon_1}=\overline{\lambda_{(r_-,r_-)\varepsilon_1}}$, as required, while the real parts of all other eigenvalues,
including $\lambda_{(r_+,r_-)\varepsilon_1}=\overline{\lambda_{(r_-,r_+)\varepsilon_1}}$, are still negative.
\end{proof}

Note that from Lemmas~\ref{LemaDeltaNeg} and \ref{LemaDeltaPos}, it follows that the first bifurcating eigenvalue for $c>0$ is precisely the non-resonant eigenvalue for $c=0$, that has the largest imaginary part.

\paragraph{\bf Acknowledgements}
The research of both authors at Centro de Ma\-te\-m\'a\-ti\-ca da Universidade do Porto (CMUP)
 had financial support from
 the European Regional Development Fund through the programme COMPETE and
 from  the Portuguese Government through the Fun\-da\-\c c\~ao para
a Ci\^encia e a Tecnologia (FCT) under the project
 PEst-C/MAT/UI0144/2011.
 A.C. Murza was also supported by the grant
SFRH/ BD/ 64374/ 2009 of FCT.

\end{document}